\documentclass[12pt, reqno]{amsart}
\usepackage{amsmath, amsthm, amscd, amsfonts, amssymb, graphicx, color}
\usepackage[bookmarksnumbered, colorlinks, plainpages]{hyperref}
\hypersetup{colorlinks=true,linkcolor=red, anchorcolor=green, citecolor=cyan, urlcolor=red, filecolor=magenta, pdftoolbar=true}

\newtheorem{theorem}{Theorem}[section]
\newtheorem{lemma}[theorem]{Lemma}
\newtheorem{proposition}[theorem]{Proposition}
\newtheorem{corollary}[theorem]{Corollary}
\theoremstyle{definition}
\newtheorem{definition}[theorem]{Definition}
\newtheorem{example}[theorem]{Example}

\theoremstyle{remark}
\newtheorem{remark}[theorem]{Remark}
\numberwithin{equation}{section}

\begin{document}

\setcounter{page}{1}

\title[Short Title]{A study of reciprocal Dunford-Pettis-like properties on Banach spaces
}

\author[M. Alikhani]{Morteza Alikhani,$^1$}

\address{$^{1}$Department of Mathematics, University of Isfahan.}
\email{\textcolor[rgb]{0.00,0.00,0.84}{m2020alikhani@yahoo.com}}







\begin{abstract}
In this article, we  study  the relationship between
$ p $-$ (V) $ subsets and $ p $-$ (V^{\ast})$ subsets of dual spaces.\ We investigate  the Banach space  $ X $ with the property that adjoint every $ p $-convergent operator $ T:X\rightarrow Y $  is weakly $ q $-compact, for every Banach space $ Y. $\ Moreover,
we define the notion of $ q $-reciprocal Dunford-Pettis$ ^{\ast} $  property  of order $ p$  on Banach spaces and obtain a characterization of Banach spaces with this property.\ Also,  the stability of
reciprocal Dunford-Pettis property of  order $ p $ for projective tensor
product is given.\
\end{abstract} \maketitle

\section{\textbf{Introduction and preliminaries}}
Numerous authors by studying localized properties, e.g.,   Dunford-Pettis sets, $ (L) $-sets, $ (V) $-sets and  $ (V^{\ast}) $-sets, showed that how these  notions can be used to study
more global structure properties.\ For instance,  Leavelle \cite{le}, by using the notion $ (L) $  sets, obtained a characterization of those Banach spaces with the reciprocal Dunford-Pettis property.\ Later on, Emmanuele \cite{e},  proved that a Banach space  $ X $ does not contain $ \ell_{1} $ if and only if any $ (L) $ subset of $ X $ is relatively compact.\
It is easy to verify that, every Dunford-Pettis subset of a dual space is an $ (L) $ subset, while the converse of implication is false.\
The relationship between $(L)$ subsets and Dunford-Pettis subsets of dual spaces obtained by Bator et al.\ \cite{bpo}.\ Recently, Li et al.\cite{ccl1} 
generalized the concepts $ (V ) $ and $ (V^{\ast} ) $ sets to the 
 to the
 $ p $-$ (V ) $ and $ p $-$ (V ^{\ast}) $ sets for $ 1\leq p\leq\infty. $\ It is easy to see that 
$ 1 $-$ (V ) $ sets are $ (V ) $ sets, $ \infty $-$ (V ) $ sets are $ (L) $ sets and $  1$-$ (V^{\ast}) $ sets are $ (V^{\ast}) $  sets.\ Note
that the definitions of $ p $-$ (V ) $ and $ p $-$ (V ^{\ast}) $ sets coincide with the definitions
of weakly $  p$-$ L $ sets and weakly $ p $-Dunford-Pettis sets given in \cite{g18}, respectively.\\ 
Inspired by the above works, we  obtain  relationship
between $ p$-$ (V ) $ subsets and $ p$-$ (V^{\ast} ) $ subsets of dual spaces ($ 1\leq p< \infty $).\ Also,
we study two  properties on Banach spaces, called the $ q $-reciprocal Dunford-Pettis property  of order $ p$ and  the $ q $-reciprocal-Dunford-Pettis$ ^{\ast} $ property  of order $ p$ ( $ 1\leq p\leq q \leq \infty$) in order to find a
necessary and
sufficient conditions, that every $ p $-$ (V) $ set in $ X^{\ast} $
 (every $ p $-$ (V^{\ast}) $ set in $ X $) is relatively weakly $ q $-compact.\ In addition, we investigate
the stability of reciprocal Dunford-Pettis$ ^{\ast} $ property  of order  $ p$ for some subspaces of  bounded linear operators.\ Note that, the our results are motivated by results in \cite{bpo} and \cite{g18}.\\

Throughout this paper  $ 1\leq p \leq \infty $ and $ 1\leq p \leq q\leq \infty,$ except for the cases where we consider other assumptions.\ Also, we suppose
  $ X,Y $ and $  Z$ are arbitrary Banach spaces, $p^{\ast}$ is the H$\ddot{\mathrm{o}}$lder conjugate of $p;$ if $ p=1,~~ \ell_{p^{\ast}} $
 plays the role of $ c_{0} .$\ The unit coordinate vector in $ \ell_{p} $ (resp.\ $ c_{0} $ or $ \ell_{\infty} $) is denoted
by $ e^{p}_{n} $
n (resp.\ $ e_{n} $).\
We denote the closed unit ball of $ X $  by $ B_{X}$ and the identity map on $X$ is denoted by $id_X.$\ The space $  X$ embeds in $ Y $ (in symbols $ X\hookrightarrow Y $) if $  X$ is isomorphic to a closed
subspace of $ Y. $\
We denote two isometrically isomorphic
spaces $ X $ and $  Y$ by $  X\cong Y.$\ Also, the topological dual of $  X$ is denoted by $ X^{\ast} $ and
 we use $ \langle   x^{\ast}, x \rangle $ or $ x^{\ast}(x) $ for the duality between $ x\in X $ and $ x^{\ast} \in X^{\ast} .$\
The space  of all bounded linear operators (compact operators)  from $ X $ into $ Y $ is denoted by $ L(X,Y) $ ( $ K(X,Y)$).\ The space of all $ w^{\ast} $-$ w $ continuous and $ w^{\ast} $-$  w$ continuous compact operators from $ X^{\ast} $ to $ Y $ will be denoted by  $ L_{w^{\ast}}(X^{\ast}, Y) $ and $ K_{w^{\ast}}(X^{\ast}, Y) ,$ respectively.\ The projective tensor product of two Banach
spaces $  X$ and $ Y $ will be denoted by $ X\widehat{\bigotimes}_{\pi}  Y .$\

A bounded linear operator $ T :X   \rightarrow Y$  is said to be completely continuous, if $ T $ maps weakly convergent sequences to norm convergent sequences.\ The set of all completely continuous operators from $ X $ to $  Y$ is denoted by $CC(X,Y) .$\ A Banach space $ X $ is said to have the Dunford-Pettis property, if for any Banach space $ Y $ every weakly compact operator $ T:X\rightarrow Y $ is completely continous.\ A Banach space $ X $ is said to have the reciprocal Dunford-Pettis property (in short, $ X $  has the $ (RDPP) $), if for any Banach space $ Y $ every completely continuous operator $ T:X\rightarrow Y $ is weakly compact \cite{gr}.\ Let us recall from \cite{An}, that a bounded
subset $ K$ of $ X $ is a Dunford-Pettis set if and only if
every weakly null sequence $(x^{\ast}_{n})_n $ in $ X^{\ast}, $ converges uniformly to zero on  the set $K.$\
A bounded subset $ K $ of $ X^{\ast} $ is called an $( L )$ set, if each weakly null sequence $ (x_{n})_{n} $ in $ X$ tends to $ 0 $ uniformly on $ K$ \cite{le}.

A sequence $(x_{n})_n$ in  $X$ is  called
weakly $p$-summable if $(x^{\ast}(x_{n}))_n\in \ell_{p}$ for each
$ x^{\ast}\in X^{\ast}.$\ The weakly  $\infty$-summable sequences  are precisely the weakly null sequences.\
 A sequence  $ (x_{n})_{n} $ in $  X$ is called weakly $ p $-convergent to  $ x\in X $ if the sequence  $ (x_{n}-x)_{n} $ is weakly $  p$-summable.\ The weakly $ \infty$-convergent sequences
are precisely the weakly convergent sequences.\ A sequence $ (x_{n})_{n} $ in X is called weakly $ p $-Cauchy if $ (x_{m_{k}}-x_{n_{k}}) _{k}$ is
weakly $ p $-summable for any increasing sequences $ (m_{k})_{k} $ and $ (n_{k})_{k} $ of positive integers.\
Note that, every weakly $  p$-convergent sequence is weakly $  p$-Cauchy, and the weakly $ \infty $-Cauchy sequences are precisely the weakly Cauchy sequences.\ We say that a subset $ K $ of $ X $ is called weakly $ p $-precompact,
if every sequence from $  K$ has a weakly $ p $-Cauchy subsequence.\ Note that the weakly $ \infty $-precompact sets are precisely the weakly precompact sets.\
A bounded linear operator $ T:X\rightarrow Y $ is called $ p $-convergent, if $  T$ maps
weakly $  p$-summable sequences into norm null sequences.\ The set of all $ p $-convergent
operators from $X$ into $Y$ is denoted by $C_p(X, Y).$\ A Banach space $ X $ has the $ p $-Schur property, if the identity operator on $ X $ is $ p $-convergent.\ A Banach space   $ X $ has the
 Dunford-Pettis property of order $p$ (in short  $ X $ has the $(DPP_{p}) $), if every weakly compact operator $  T:X \rightarrow Y$ is $ p $-convergent,
for any Banach space $  Y.$\ A bounded subset $ K $ of $ X^{\ast} $ is a $ p $-$ (V ) $  set, if $ \displaystyle \lim_{n\rightarrow\infty}\displaystyle \sup_{x^{\ast}\in K}\vert x^{\ast}(x_{n}) \vert =0,$
for every weakly $ p $-summable sequence $ (x_{n})_{n} $ in $ X. $\ A  bounded subset $ K $ of $ X $ is a $ p $-$ (V^{\ast} ) $ set, if $ \displaystyle \lim_{n\rightarrow\infty}\sup_{x\in K}\vert x^{\ast}_{n}(x) \vert =0,$
for every weakly $ p $-summable sequence $ (x^{\ast}_{n})_{n} $ in $ X^{\ast}. $\
A Banach space $ X $ has
Pelczy\'{n}ski's  property $ (V) $ of order $ p$ ( in short $ X $ has the $ p $-$ (V) $ property), if every $ p $-$ (V) $ set in $
X^{\ast}$ is relatively weakly compact.\ A Banach space $ X $ has   Pelczy\'{n}ski's  property $ (V^{\ast}) $ of order $  p$ (in short $ X $ has the $ p $-$ (V^{\ast}) $ property), if every $ p $-$ (V^{\ast}) $ set in $  X^{\ast}$ is relatively weakly compact.\ Let us recall from \cite{sk}, that $ \ell_{p}(X)$ denote the set of all sequences $ (x_{n})_{n} $ in $ X $ such that $  \displaystyle\sum_{n=1}^{\infty}\Vert x_{n}\Vert^{p}<\infty.$
A set $ K\subset X $ is said to be relatively $ p $-compact if there is a
sequence$ (x_{n})_{n} $ in $ \ell_{p}(X)$ such that $ K\subset \lbrace   \displaystyle \sum_{n=1} ^{\infty} \alpha_{n} x_{n} :(\alpha_{n})\in B_{\ell_{p^{\ast}}}\rbrace.$\ An operator
$ T\in L(X,Y) $ is said to be $  p$-compact if $  T(B_{X})$ is a relatively $ p $-compact set in $ Y. $\

A bounded subset $ K $ of $  X$ is said to be relatively weakly $  p$-compact (resp.\ weakly $  p$-compact) provided that every sequence in $ K $ has a weakly $  p$-convergent subsequence
with limit in $  X$ (resp.\ in $K$).\ Note that, the weakly $\infty$-compact sets are precisely the weakly compact.\ A bounded linear operator $ T:X\rightarrow Y $  is called weakly $ p $-compact if $ T(B_{X}) $ is relatively weakly $ p$-compact.\ The set of  all weakly $ p $-compact operators
$ T:X\rightarrow Y$  is denoted  by $ W_{p}(X,Y) .$\\
The reader is referred to \cite{AlbKal} for any unexplained notation or terminology.\
\section{ main results}
 Suppose that $ K $ is a bounded subset of $ X$  and
$ B(K) $ is the Banach space of all bounded real-valued functions
defined on $ K$, provided with the superemum norm.\ The natural evaluation map $ E : X^{\ast}\rightarrow B(K) $ defined by
$ E(x^{\ast})(x) = x^{\ast}(x)$ has been used by many authors to study properties of $ K.$\ Similarly, if $ K $ is a bounded
subset of $ X^{\ast}, $ the natural evaluation map $ E_{X} :X \rightarrow B(K) $  defined by $ E_{X}(x)(x^{\ast}) =x^{\ast}(x)$ (for instance, see \cite{bpo,e}).\\
 Inspired
by Theorem 3.1 of \cite{bpo}, we obtain some characterizations of notions $ p $-$ (V) $
sets and $ p $-$ (V^{\ast}) $ sets  which will be used in the
sequel.

\begin{lemma}\label{l1} The following statements hold:\\
$ \rm{(i)} $  If $ T \in L(X,Y) ,$ then
$ T^{\ast}(B_{Y^{\ast}}) $ is a $ p $-$ (V) $ subset of $ X^{\ast} $ if and only if
$ T $ is $ p $-convergent.\\
$ \rm{(ii)} $  If $ T \in L(X,Y) ,$ then
$ T(B_{X}) $ is a $ p $-$ (V^{\ast}) $ subset of $ Y$ if and only if
$ T^{\ast} $ is $ p $-convergent.\\
$ \rm{(iii)} $
A bounded subset $ K $ of $ X^{\ast} $ is a  $ p $-$ (V) $ set  if and only if $E_{X} :X \rightarrow B(K) $  is $ p $-convergent.\\
$ \rm{(iv)} $
A bounded subset $ K $ of $  X$ is a $ p $-$ (V^{\ast}) $ set if and only if $ E : X^{\ast}\rightarrow B(K) $ is $ p $-convergent.\\
$ \rm{(v)} $ A bounded subset $ K $ of $ X $ is a $ p $-$ (V^{\ast}) $  set if and only if there is a Banach space
$ Y $ and an operator $ T:Y\rightarrow X $ so that $ T $ and $ T^{\ast} $ are $ p $-convergent
and $ K\subseteq T(B_{Y}) .$\
\end{lemma}
\begin{proof}
The assertions $ \rm{(i)} $ and $ \rm{(ii)} $ when $ 1\leq p<\infty $ are in {\rm (\cite[Theorem 14]{g18})}, while for $ p=\infty $ they are in {\rm (\cite[Theorem 3.1]{bpo})}.\ Hence, we only prove $ \rm{(iii)} ,\rm{(iv)}$ and $ \rm{(v)} .$\ Note that  we adapt the proofs $ \rm{(i)},\rm{(ii)} $  and $ \rm{(iii)} $ of  {\rm (\cite[Theorem 3.1]{bpo})}.\\
$ \rm{(iii)} $ Suppose that $ K $ is a bounded subset of $ X^{\ast}. $\ Therefore  $ E_{X} $ is  $  p$-convergent if and only if  $ \Vert E_{X}(x_{n}) \Vert\rightarrow 0$ for each weakly $ p $-summable sequence $ (x_{n})_{n} $ in $ X $ if and only if
$$ \lim_{n} (\sup \lbrace \vert x^{\ast}(x_{n})\vert :x^{\ast} \in K\rbrace )=0$$ for each weakly $ p $-summable sequence $ (x_{n})_{n} $ in $ X $ if and only if   $K$ is a $ p $-$ (V) $ set.\\
$ \rm{(iv)} $ Suppose that $ K $ is a bounded subset of $ X $ and $ E$ is a $ p $-convergent operator.\ Thus $ E^{\ast} $ maps the unit ball of $ B(K) ^{\ast},$ to a
$ p $-$ (V) $ set in  $ X^{\ast\ast} .$\ However, if $ k \in K $ and $  \delta_{k}$ denotes the point mass at $ k, $ then
$ E^{\ast}(\lbrace\delta_{k}:k \in K\rbrace) = K, $
and  so $ K $ is a $ p $-$ (V) $ set in $ X^{\ast\ast} .$\ Hence $ K $ is a $ p $-$ (V^{\ast}) $ set in $ X. $\\
Conversely, suppose that $  K$ is a $ p $-$ (V^{\ast}) $ set in $  X$, and let $ E $
be the evaluation map.\ If $ (x^{\ast}_{n})_{n} $ is a weakly $ p $-summable sequence in $ X^{\ast}, $
 then
 $$\lim_{n}\Vert E(x^{\ast}_{n})\Vert=\lim_{n} (\sup\lbrace \vert x^{\ast}_{n}(x) \vert :x \in K\rbrace)=0, $$
and so $  E$ is a $ p$-convergent operator.\\
$ \rm{(v)} $ Suppose that $ K $ is a $ p $-$ (V^{\ast}) $ set in $ X $  and $ Y=\ell_{1} (K),$
 Define $ T:Y\rightarrow X $ by  $ T(f) =\sum_{k\in K} f(k)k $ for $f \in \ell_{1}(K) .$\ It is clear that
 $ T $ is a bounded linear operator such that $ K\subseteq T(B_{\ell_{1}(K)}).$\ Since $ \ell_{1} (K)$ has the $ p $-Schur property, the operator $  T$ is $ p $-convergent.\ Moreover, $ T^{\ast} $ is the
evaluation map $ E,$ and $ T^{\ast} $ is $ p $-convergent by $\rm{ (iv)}. $\
\end{proof}
It is easy to verify that, for each  $ 1\leq p\leq \infty, $ every  $ p $-$ (V^{\ast} ) $ subset of dual space is a  $ p $-$ (V ) $ set, while  the converse of implication is false.\
The following theorem continues our study of the relationship between
$ p $-$ (V) $ subsets and $ p $-$ (V^{\ast}) $ subsets of dual spaces.\
\begin{theorem}\label{t1}
Every $ p$-$ (V ) $ subset of $ X^{\ast} $ is a $ p $-$ (V^{\ast} ) $ set in $ X^{\ast} $ if and only if
$ T^{\ast\ast} $ is a $ p$-convergent operator whenever $ Y $ is an arbitrary Banach space and
$ T : X \rightarrow Y $ is a $ p $-convergent operator.
\end{theorem}
\begin{proof} We adapt the proof of  {\rm (\cite[Theorem 3.4]{bpo})}.\
Suppose that
$ T:X\rightarrow Y $ is a $ p $-convergent operator.\ The part (i) of Lemma of \ref{l1}, yields that $ T^{\ast}(B_{Y^{\ast}} ) $ is a
$ p $-$ (V ) $ set.\ By the hypothesis $ T^{\ast}(B_{Y^{\ast}} ) $ is a $ p $-$ (V^{\ast} ) $ set.\ By applying the part (ii) of Lemma \ref{l1}, we see that $ T^{\ast\ast} $ is a $p$-convergent operator.\

Conversely, suppose that
$ K $ is a $ p $-$ (V ) $ subset of $ X^{\ast}. $\ The part (iii)  of Lemma \ref{l1}, implies  that $ E_{X} $ is $ p $-convergent.\ Therefore, by the hypothesis,  $ E_{X}^{\ast\ast} $ is $ p$-convergent.\ Hence, if $ S $ denotes the unit ball of $ B(K)^{\ast}, $  then
$ E^{\ast}_{X}(S) $ is a $ p $-$ (V^{\ast} ) $ set.\ Since $K \subset E_{X}^{\ast}(S), $
$ K $ is a $ p $-$ (V^{\ast} ) $ set in $ X^{\ast}. $\
\end{proof}
\begin{corollary}\label{c1} {\rm (\cite[Theorem 3.4]{bpo})} Let $ X $ be a Banach space.\
Every $ (L ) $ subset of $ X^{\ast} $ is a  Dunford-Pettis  set in $ X^{\ast} $ if and only if
$ T^{\ast\ast} $ is completely continuous  whenever $ Y $ is an arbitrary Banach space and
$ T : X \rightarrow Y $ is a completely continuous operator.
\end{corollary}
\begin{definition} \label{d1} Suppose that $ 1\leq p\leq q\leq \infty. $\ We say that a Banach space $ X $ has  the $ q $-reciprocal
Dunford-Pettis  property
 of order $ p$ (in short $ X $  has the $ q $-$ (RDPP)_{p} $), if  the adjoint every $ p $-convergent
operator from $ X $ to $ Y $ is weakly $ q $-compact, for every Banach space $ Y. $\
\end{definition}
  The $ \infty $-$ (RDPP)_{\infty} $ is  precisely the $ (RDPP) $ and $ \infty $-$ (RDPP)_{p} $ is  precisely the reciprocal
Dunford-Pettis  property of order $ p $ (in short $ (RDPP)_{p}$) introduced by Ghenciu \cite{g18}.\ Note that the property $ (RDPP)_{p} $ coincides with the property $ (V) $ of order $ p $ introduced
 by  Li et.al.(see Definition at page 443 and Theorem 21 in \cite{g18}  and Theorem 2.4 in \cite{ccl1}).\
 \begin{proposition}\label{p1}
 A Banach space $  X$ has the $ q $-$ (RDPP)_{p} $ if and only if the adjoint  of every evaluation map $ E_{X}:X\rightarrow B(K)$ associated with a subset $ K $ of $ X^{\ast} ,$
 is  weakly $ q $-compact whenever it is $ p $-convergent.\
\end{proposition}
\begin{theorem}\label{t2}  A Banach space $ X $ has the $ q $-$ (RDPP)_{p} $ property if and only if every $ p $-$ (V) $ subset of $ X^{\ast} $ is relatively weakly $ q $-compact.\
\end{theorem}
\begin{proof}
Suppose that $ X $ has the $ q $-$ (RDPP)_{p} $ and let $ K $ be a $ p $-$ (V) $ subset of $ X^{\ast} .$\
Therefore, $ E_{X}$ is $ p$-convergent, so, by the hypothesis, $ E_{X}^{\ast} $ is weakly $ q $-compact.\ Since, $ K=\lbrace E_{X}^{\ast}(\delta_{x^{\ast}}):x^{\ast}\in K\rbrace \subseteq E_{X}^{\ast}(S),$
where $ S $ is the unit ball in $ B(K)^{\ast}, $ it is relatively weakly $ q $-compact.\\
Conversely,  if $ T: X \rightarrow Y $ is a $ p $-convergent operator.\ From part (i) of Lemma \ref{l1}, $ K = T^{\ast}(B_{Y^{\ast}}) $ is a $ p$-$ (V) $ set in $ X^{\ast}. $\ Therefore, $ K $ is relatively weakly $ q $-compact and so $ T^{\ast} $ is weakly $ q $-compact.\
\end{proof}

A bounded linear operator $ T:X\rightarrow Y $ is said to be  strictly singular if there is no infinite dimensional
subspace $ Z\subseteq X $ such that $ T_{\vert Z} $ is an isomorphism onto its range (see \cite{AlbKal}, Definition 2.1.8).\
By {\rm (\cite[Proposition 2.16]{le})}, if $ T: X\rightarrow Y $ is  completely continuous  and $ X\in(RDPP),$  then $ T $ is strictly singular.\
\begin{proposition}\label{p2}
 Suppose that $ T: X\rightarrow Y $ is a $ p$-convergent operator.\ If $ X $ has the $ p $-$ (RDPP)_{p} ,$ then $ T $ is strictly singular.\
\end{proposition}
\begin{proof}  Since $  X$ has the $ p $-$ (RDPP)_{p}, $  $ T \in C_{p}(X,Y)\cap W_{p}(X,Y).$\ Thus, an application of  Corollary 2.23 in \cite{dm} shows that, $  T$ is strictly singular.
\end{proof}

\begin{corollary}\label{c2}
Suppose that $ X$ has the $ q $-$(RDPP)_{p}. $\ The following statements hold:\\
$ \rm{(i)} $ Every quotient space of $ X $ has the same property.\\
$ \rm{(ii)} $ If $ X $ has the $ p$-Schur property, then $ B_{X^{\ast}} $ is a weakly $ q $-compact.\\
$ \rm{(iii)} $ If $X$ has the $(DPP_{q}),$ then every $ p $-$ (V) $ subset of $ X^{\ast} $ is $ q$-$ (V) $ set.\

\end{corollary}
\begin{proof}
$ \rm{(i)} $ Suppose that $ X $ has the  $ q $-$(RDPP)_{p} ,$ $ Z$ is a quotient of $ X$ and $ Q : X \rightarrow Z $
is a quotient map.\ Let $ T : Z \rightarrow Y $ be a $ p $-convergent operator.\ Therefore $ T\circ Q : X \rightarrow Y $
is $ p $-convergent, and thus $ (T\circ Q)^{\ast} $ is weakly $ q $-compact.\ Since $ Q^{\ast} $ is an
isomorphism and $ Q^{\ast}(T^{\ast}(B_{Y^{\ast}} )) $ is relatively weakly $ q $-compact, $ T^{\ast}(B_{Y^{\ast}} ) $ is relatively
weakly $ q $-compact.\\
$ \rm{(ii)} $ Since $ X $ has the $ p $-Schur property,  $ B_{X^{\ast}} $ is a $ p $-$ (V) $ set, and thus weakly $ q $-compact, since $ X $ has the $ q$-$ (RDPP)_{p}. $\\
$ \rm{(iii)} $ Suppose that $ K $ is a $ p $-$ (V) $ subset of $ X^{\ast} .$\ Since $ X $ has the $ q $-$(RDPP)_{p}, $ Theorem \ref{t2} implies that $ K $ is a relatively  weakly $ q $-compact.\ Hence $ K $ is a relatively weakly compact.\ Apply {\rm (\cite[Theorem 3.11]{ccl})}.\
\end{proof}
The James $ p $-space $ J_{p}~(1 < p < \infty) $ is the (real) Banach space
of all sequences $ (a_{n})_{n} $ of real numbers such that $ \lim_{n\rightarrow\infty}a_{n}=0 $ and norm on $ J_{p} $ is given by the formula
\begin{center}
$\Vert a_{n}\Vert_{pv}=\sup\lbrace (\displaystyle \sum_{j=1}^{m}\vert a_{i_{j-1}}-a_{i_{j}}\vert^{p})^{\frac{1}{p}}:1\leq i_{0}< i_{1}<\cdot\cdot\cdot<i_{m},m\in\mathbb{N} \rbrace.$
\end{center}
\begin{corollary}\label{c3}
 The James $ 2 $-space $ J_{2}$ does not have the $ 2 $-Schur property.
\end{corollary}
\begin{proof}
Suppose that $ J_{2}$ has the $ 2 $-Schur property.\ Since it has $ (RDPP)_{2},$ by Corollary \ref{c2}, $ B_{J^{\ast}} $ would be weakly compact and then  $ J_{2}$ would be reflexive space which is a contradiction.
\end{proof}

\begin{example}\rm \label{e1}
 $ \rm{(i)} $ $ \ell_{2} $ has the  $ 2$-$ (RDPP)_{1} .$\ Indeed,
we know that $ \ell_{2} $ contains no copy of $ c_{0}. $\ Therefore, $ \ell_{2} $ has the
$ 1 $-Schur property;  {\rm (\cite[Theorem 2.4]{dm})}.\ Hence $ B_{\ell_{2}} $ is a $ 1 $-$ (V) $ set.\ Also by
 {\rm (\cite[Proposition 4]{C1})}, the closed unit ball of $ \ell_{2} $ is a weakly $ 2 $-compact set.\
  Now, let $ K $ be a $ 1 $-$ (V) $ subset of $ \ell_{2}. $\ Since every
$ 1 $-$ (V) $ subset of dual space is bounded,  we may assume that $ K \subseteq \alpha B_{\ell_{2}},$ for some $ \alpha >0. $\ Hence by Theorem \ref{t2},
 $ \ell_{2} $ has the $ 2$-$ (RDPP)_{1} .$\\
  $ \rm{(ii)} $ It is known that $ L_{1}([0,1]) $   contain no copy of $ c_{0} .$\ Therefore $ L_{1}([0,1]) $ has the $ 1 $-Schur property.\ Hence, the part (ii) of Corollary \ref{c2}, implies that  $ L_{1}([0,1]) $ does not have the $ 2$-$ (RDPP)_{1}. $\
 \end{example}
Let us recall from \cite{AlbKal}, that  the finite regular Borel signed measures on the compact space $ K $ is denoted
by $  C(K)^{\ast}=M(K).$\
\begin{corollary}\label{c4}
 If $ K $ is a compact Hausdorff space, then every $ p $-$ (V) $ subset
of $ M(K)$ is a $ p $-$ (V^{\ast}) $ set in $ M(K) .$\
\end{corollary}
\begin{proof}
We repeat with the obvious changes the proof of Corollary 3.5 in \cite{bpo}.\
Suppose that $ K $ is a compact Hausdorff space, $ Y $ is a Banach
space and $ T : C(K) \rightarrow Y $ is a $ p $-convergent operator.\ Since  $ C(K) $ has the $ p $-$ (V)$ property, it  has the $ (RDPP)_{p} $ by Definition 2.1 in \cite{ccl1}.\ Therefore $ T $ is weakly
compact and so, $ T^{\ast\ast} $ is weakly compact.\ On the other hands, $ M(K)^{\ast} $ is also
a continuous functions space.\ Therefore $ M(K)^{\ast}$ has the $(DPP_{p}) $ and so, $ T^{\ast\ast} $ is $ p $-convergent.\ Hence, Theorem \ref{t1} implies that, every $ p $-$ (V) $ subset of
$ M(K)$ is a $ p $-$ (V^{\ast}) $ set in $ M(K)$
\end{proof}
\begin{proposition}\label{p3}
 The Cartesian product $ X\times Y $ has the $ (RDPP)_{p} $
if and only if $ X $ and $ Y $ have the same property.
\end{proposition}
\begin{proof} Since $  X$ and $  Y$ are  quotients of $ X\times Y, $
the necessity of the result follows from {\rm (\cite[Corollary 23]{g18})}.\
Now, suppose that $ X $ and $ Y $ have the $ (RDPP)_{p} .$\ For arbitrary Banach space $ Z, $ let the operator $ T:X\times Y\rightarrow Z$ be $ p $-convergent.\ We show that $ T^{\ast}$ is weakly compact.\ For this purpose,  we define $ T_{1}:X \rightarrow Z $ by $T_{1} (x)=T(x,0)$ and $ T_{2}:Y\rightarrow Z $ by $T_{2} (y)=T(0,y).$\
It is clear that $ T_{1} $ and $ T_{2} $ are $ p $-convergent.\ Since, $ X $ and $ Y $ have the $ (RDPP)_{p},$ by {\rm (\cite[Theorem 21]{g18})} $ T^{\ast}_{1} $ and $ T^{\ast}_{2} $ are weakly compact operators and so $ T_{1} $ and $ T_{2} $ are weakly compact operators.\ If $ (x_{n},y_{n})_{n} $ is a bounded sequence in $ X\times Y $, then $ (x_{n})_{n} $ and $ (y_{n})_{n} $ are bounded sequences in $ X$ and $Y,$ respectively.\  Hence,
 $ (T_{1}(x_{n}), T_{2}(y_{n}))_{n} $ have weakly convergent subsequence in $ Z\times Z $\ Therefore,  $(T (x_{n},y_{n}))_{n} $ has a weakly convergent subsequence in $ Z. $\  Hence, $ T \in W(X\times Y, Z)$ and so, $ T^{\ast} $ is weakly compact.\ Applying {\rm (\cite[Theorem 21]{g18})} implies that $ X\times Y $ has the $ (RDPP)_{p}. $\
\end{proof}

\begin{lemma}\label{l2}
Suppose that $ (x_{n})_{n} $ is a weakly $ p $-summable sequence in $ X $ and  let $ (y_{n})_{n} $ be a bounded sequence in $ Y. $\ If the
adjoint of every bounded linear operator $ T: X\rightarrow Y^{\ast} $ is $ p $-compact, then $ (x_{n}\otimes y_{n})_{n} $ is weakly $ p $-summable
in $ X \widehat{\bigotimes}_{\pi} Y. $
\end{lemma}
\begin{proof} Suppose that $ T: X\rightarrow Y^{\ast} $ is a bounded linear operator such that $ T^{\ast} $ is $ p $-compact.\ Hence {\rm (\cite[Proposition 5.3 (c)]{sk})}, implies that every $T\in (X \widehat{\bigotimes}_{\pi} Y)^{\ast}=L(X,Y^{\ast}) $ is $p$-summing operator.\ Now, let\\ $M:=\displaystyle\sup_{n}\lbrace\Vert y_{n}\Vert :n\in \mathbb{N}\rbrace$ then for each $T\in ( X\widehat{\bigotimes}_{\pi}Y)^{\ast}, $ we have:
$$(\displaystyle\sum_{i=1}^{\infty}\vert \langle T, x_{n}\bigotimes y_{n}\rangle\vert^{p})^{\frac{1}{p}}\leq M\Vert T(x_{n})\Vert _{\ell_{p}(Y)}<\infty
$$ Hence, $ (x_{n}\bigotimes y_{n})_{n} $ is a weakly $p$-summable sequence in $X \widehat{\bigotimes}_{\pi}Y.$\end{proof}

Note that, there are examples of Banach
spaces X and Y such that $ X \widehat{\bigotimes}_{\pi} Y $ has Pelczy$\acute{n} $ski's property $ (V) $ of order
$ p .$\ For example, let $  1< q^{\ast}< p <\infty.$\ It is easily verified that, $ L(\ell_{p},\ell_{q^{\ast}})=(\ell_{p}\widehat\bigotimes \ell_{q})^{\ast}$ is reflexive.\ Hence $ \ell_{p}\widehat{\bigotimes } \ell_{q}$ is reflexive, and so has Pelczy$\acute{n} $ski's property $ (V) $ of order
$ p. $\ Thus the spaces $X=\ell_{p} $ and $ Y=\ell_{q} $ are as desired.
\begin{theorem}\label{t3}
If $ X $ and $ Y $ have Pelczy$\acute{n} $ski's property $ (V) $ of order
$ p $ and the
adjoint of every bounded linear operator $ T: X\rightarrow Y^{\ast} $ is $ p $-compact, then $ X \widehat{\bigotimes}_{\pi} Y $ has the same property.
\end{theorem}
\begin{proof}
Let $ K $ be a $ p $-$ (V) $ subset of $ (X \widehat{\bigotimes}_{\pi} Y)^{\ast}=L(X,Y^{\ast}) .$\ We claim that $ K $ is relatively weakly compact.\ We show that the conditions $ \rm{(i)} $ and $ \rm{(ii)} $ of {\rm (\cite[Theorem 4 ]{g24})} are true.\
Let $ (T_{n})_{n} $ be a sequence in $ K .$\ If $ y^{\ast\ast}\in Y^{\ast\ast} , $ it is enough to show that $ \lbrace T^{\ast}_{n}(y^{\ast\ast}) : n\in \mathbb{N}\rbrace $ is a $ p $-$(V) $ subset of $ X^{\ast}. $\ For this purpose,
suppose that $ (x_{n})_{n} $ is a weakly $ p $-summable sequence in $ X. $\ For $ n\in \mathbb{N}, $ we have:
$$ \vert\langle T^{\ast}_{n}(y^{\ast\ast}), x_{n} \rangle \vert =\vert \langle y^{\ast\ast}
, T_{n} (x_{n}) \rangle\vert\leq \Vert y^{\ast\ast} \Vert \Vert T_{n} (x_{n}) \Vert.$$\
We claim that $ \Vert T_{n}(x_{n}) \Vert\rightarrow 0. $\ Suppose that
$ \Vert T_{n}(x_{n} )\Vert \not\rightarrow 0. $\ Without loss of
generality we assume that $ \vert T_{n}(x_{n}) ( y_{n} )\vert >\varepsilon $
for some sequence $ (y_{n})_{n} $ in $ B_{Y} $ and
some $ \varepsilon >0. $\ Lemma \ref{l2} implies that $ (x_{n}\otimes y_{n})_{n} $ is a weakly $ p $-summable sequence in $ X \widehat{\bigotimes}_{\pi} Y. $\ Since
$ \lbrace T_{n}: n\in \mathbb{N} \rbrace $ is a $ p $-$ (V) $ set, we have :
\begin{center}
$ \vert \langle T_{n}( x_{n}), y_{n}\rangle \vert=\vert \langle T_{n},x_{n}\otimes y_{n} \rangle \vert\rightarrow 0, $
\end{center}
which is a contradiction.\ Hence $ \lbrace T^{\ast}_{n}(y^{\ast\ast}) : n\in \mathbb{N}\rbrace $ is a $ p $-$(V) $ subset of $ X^{\ast}. $\
Therefore this subset is relatively weakly compact, since $ X $ has the $ p$-$ (V) $ property.\
Now, let $ x \in X. $\ By
an argument similar,  $ \lbrace T_{n}(x):n\in\mathbb{N} \rbrace$ is a $ p $-$(V) $ subset of $ Y^{\ast}, $ and so $ \lbrace T_{n}(x):n\in\mathbb{N} \rbrace$ is relatively weakly compact for all $ x\in X .$\ Hence $ K $ is relatively weakly compact.
\end{proof}

A direct consequence of Theorem \ref{t3} is the following corollary which is the $  p$-version of {\rm (\cite[Theorem 2.7]{e1})}.
\begin{corollary}\label{c5}
Suppose that $ B_{X} $ is  weakly $ p $-precompact  and $ Y $ has the $ (RDPP)_{p}.$\  If the adjoint of every bounded linear operator $ T: X\rightarrow Y^{\ast} $ is $ p $-compact, then
$ X \widehat{\bigotimes}_{\pi} Y $ has the  $ (RDPP)_{p}.$\
\end{corollary}
As an immediate consequence of the Theorem 2.6  in \cite{ccl}, we can conclude that
the following result.
\begin{proposition} \label{p4}
If $ B_{X} $ is  weakly $ p $-precompact, then the following statements holds:\\
$ \rm{(i)} $ Every $ p$-$ (V) $ subset of $ X^{\ast} $ is relatively compact.\\
$ \rm{(ii)} $ $  X $ has the $ (RDPP) _{p}.$\
\end{proposition}
\begin{definition} \label{d2} Let $ 1\leq p\leq q \leq \infty.$\ We say that $ X $ has the $ q $-reciprocal
Dunford-Pettis$ ^{\ast} $ property of order $ p$ (in short $ X $ has the $ q$-$ (RDP^{\ast}P)_{p} $), if for each Banach space $ Y,$ every bounded linear operator $T : Y \rightarrow X $ is  weakly $ q $-compact, whenever $ T^{\ast}: X^{\ast}\rightarrow Y^{\ast}$  is $ p $-convergent.
\end{definition}
The $ \infty $-$ (RDP^{\ast}P)_{\infty} $ is  precisely the $ (RDP^{\ast}P) $ and $ \infty $-$ (RDP^{\ast}P)_{p} $ is  precisely the  $ (RDP^{\ast}P)_{p} $ introduced by Ghenciu
(see Definition at page 444 and Theorem 15 of \cite{g18}).\ Note that $ (RDP^{\ast}P)_{p} $ coincide with the property $ (V^{\ast}) $ of order $ p $ in \cite{ccl1}.\

\begin{theorem}\label{t4} A Banach space $ X $ has the $ q $-$ (RDP^{\ast}P)_{p}$ if and only if
every $ p $-$ (V^{\ast}) $ subset of $ X $ is relatively weakly $ q $-compact.\
\end{theorem}
\begin{proof}
We adapt the proof of {\rm (\cite[Theorem 15]{g18})}.\
 Let $ T : Y\rightarrow X $ be a bounded linear operator such that $ T^{\ast}: X^{\ast} \rightarrow Y ^{\ast} $ is $ p $-convergent.\ From part (ii) of Lemma \ref{l1},
 $ T (B_{Y} ) $ is a $ p $-$ (V^{\ast}) $ set and so  $ T (B_{Y} ) $ is relatively weakly $ q$-compact.\ Hence, $ T $ is weakly $ q $-compact.\\
Conversely, let $ K$ be a $ p $-$ (V^{\ast}) $ subset of $ X$ and let $ (x_{n})_{n} $ be a
sequence in $ K. $\ Let
 $ T :\ell_{1} \rightarrow X $ be defined by $T(b) =\displaystyle\sum_{i} b_{i}x_{i}.$\ It is clear  that $ T^{\ast}:X^{\ast}\rightarrow\ell_{\infty},~ T^{\ast}(x^{\ast}) =(x^{\ast}(x_{n}))_{n}.$\
Suppose $ (x^{\ast}_{n})_{n} $
is a weakly $ p$-summable sequence in $ X^{\ast}. $\ Since $ K$
is a $ p $-$ (V^{\ast}) $ set, $ \Vert T^{\ast}(x^{\ast}_{n})\Vert=\displaystyle\sup_{i}\vert x_{n} ^{\ast}(x_{i})\vert\rightarrow 0.$\
Therefore $ T^{\ast} $ is $ p $-convergent and thus $ T $ is weakly $ q $-compact.\ Let $ (e^{1}_{n})_{n} $
be the unit
basis of $ \ell_{1}. $\ Then $ (x_{n})_{n}=(T(e^{1}_{n}))_{n}$
has a weakly $ q $-convergent subsequence.\
\end{proof}

\begin{corollary}\label{c6}
The following statements hold:\\
$ \rm{(i)} $ Suppose that $ Y$ is a closed subspace of $ X^{\ast} $ and $ X $ has the $ q $-$ (RDPP)_{p} .$\ Then $ Y $ has the
$ q $-$ (RDP^{\ast}P)_{p} .$\\
$ \rm{(ii)} $ If $ Y^{\ast} $ has the $ q $-$ (RDPP)_{p} ,$ then $ Y $ has the $ q $-$ (RDP^{\ast}P)_{p} .$\\
$ \rm{(iii)} $ Every $ L_{1}(\mu) $ space has the $ (RDP^{\ast}P)_{p}. $\
\end{corollary}
\begin{proof}
$ \rm{(i)} $ Let $ K $ be a $ p $-$ (V^{\ast}) $ subset of $ Y. $\ Then $ K$ is a $p $-$ (V^{\ast}) $ subset
of $ X^{\ast}$ and thus a $ p $-$ (V) $ subset of $ X^{\ast}.$\
Hence, $ K$ is relatively weakly $ q $-compact.\ Therefore, $ Y $ has the $ q $-$ (RDP^{\ast}P)_{p} .$\\
$ \rm{(ii)} $ Consider $ Y $ a closed subspace of $ Y^{\ast\ast} $ and apply $ \rm{(i)} .$\\
$ \rm{(iii)} $ Let $ (\Omega,\Sigma,\mu) $
be any $ \sigma $-finite measure space.\ It is well known
that $ L^{\ast}_{1}(\mu)=L_{\infty}(\mu) $ is isometrically isomorphic to the algebra $ C(K) $ for some compact Hausdorff space $  K $ {\rm (\cite[Theorem 4.2.5]{AlbKal})}.\ Since $ C(K) $ spaces has  the $ (RDPP)_{p},$ we apply $ \rm{(i)}. $
\end{proof}

\begin{lemma}\label{l3}\cite{RP}
Let $ Y$ be a separable subspace of $ X. $\ Then there is a separable subspace $ Z$ of $ X $ that
contains $ Y $ and an isometric embedding $ J:Z^{\ast} \rightarrow X^{\ast}$
such that
$\langle J(z^{\ast}), z\rangle=\langle z^{\ast}, z \rangle $
for each $ z \in Z $
and $ z^{\ast}\in Z^{\ast} .$\
\end{lemma}
\begin{theorem}\label{t5}
$ \rm{(i)} $ If $ X $  has the $ (RDP^{\ast}P)_{p},$  then it has the $ 1 $-Schur property. \\
 $ \rm{(ii)} $ A Banach space $ X$ has the $ (RDP^{\ast}P)_{p} $ if and only if any
closed separable subspace of $ X $ has the same property.
\end{theorem}
\begin{proof}
$ \rm{(i)} $ If $ X $ has the $ (RDP^{\ast}P)_{p} , $ then it contains no copy of $ c_{0}, $ since, 
  consider the sequence $ x_{n}=e_{1}+\cdot\cdot\cdot+e_{n} $ in $ c_{0}, $
where $ (e_{n})_{n} $ is the unit vector basis.\ Obviously $ \lbrace x_{n}: n\in \mathbb{N}\rbrace $ is  $ p $-$ (V^{\ast}) $ set in $ c_{0} $ which is not relatively weakly compact and so,
 $ c_{0} $ does not have the $p $-$(V^{\ast}) $ property.\ Therefore, $ X $ contain no copy of $ c_{0} .$\ Then by Theorem 2.4 in \cite{dm}, $ X $ has the $ 1 $-Schur property.\\
 $ \rm{(ii)} $ We adapt the proof of Theorem 3.3 in \cite{g6}.\ Suppose that $ X $ has the $ (RDP^{\ast}P)_{p} $ and $ Y $ is a closed separable subspace of $ X .$\ Then  any $ p $-$ (V^{\ast}) $ subset
of $ Y$ is also a $ p $-$ (V^{\ast}) $ set in $ X.$\ Hence, $ Y $ has the $ (RDP^{\ast}P)_{p} . $\
Conversely, suppose that any closed separable subspace of $ X $ has the $ (RDP^{\ast}P)_{p}  $  and let $ K $
be a subset of $ X $ which is not relatively weakly compact.\ We show that $ K $ is not a $ p $-$ (V^{\ast}) $ set in
$ X. $\ For this purpose, let $ (x_{n})_{n} $ be a sequence in $ K$ with no weakly convergent subsequence and let $ Y=[x_{n}] $ be the
closed linear span of $ (x_{n})_{n} .$\ Note that $ Y $ is a separable subspace of $ X. $\ By Lemma \ref{l3}, there is a
separable subspace $ Z $ of $ X $ and an isometric embedding $ J:Z^{\ast}\rightarrow X^{\ast} $ which satisfy the conditions
of Lemma \ref{l3}.\ Without loss generality, we assume that $ Z $ is closed.\ Therefore, by  our hypothesis $ Z $ has the $ (RDP^{\ast}P)_{p} .$\ Thus, $ (x_{n})_{n} $ is not a $ p $-$ (V^{\ast}) $ subset of $ Z. $\ Hence, there is a weakly $ p $-summable sequence $ (z^{\ast}_{n})_{n} $
in $ Z^{\ast} $ and a subsequence $ (x_{k_{n}}) $ of $ (x_{n})_{n} ,$ which we still
denote by $ (x_{n})_{n}, $ such that $ \langle z^{\ast}_{n}, x_{n}\rangle =1 $
for each $ n\in \mathbb{N}. $\
Let $ x^{\ast}_{n}=J(z^{\ast}_{n}) $
for each $ n\in \mathbb{N}. $\ It is clear that $ (x^{\ast}_{n})_{n} $
is weakly $ p $-summable in $ X^{\ast} $ and for each $ n, $
$ x^{\ast}_{n}(x_{n})=J(z^{\ast}_{n})(x_{n})=z^{\ast}_{n}(x_{n})=1. $\
Therefore, $ K $ is not a $ p $-$ (V^{\ast}) $ subset of $ X. $
 \end{proof}

Let $ (X_{n})_{n} $ be a sequence of Banach spaces and $ 1 \leq r < \infty .$ We denote
by $ (\displaystyle \sum_{n=1}^{\infty}\oplus X_{n})_{r}$ the space of all vector-valued sequences $ x = (x_{n})_{n} $ with $ x_{n}\in X_{n} ~(n\in\mathbb{N}), $ for which
\begin{center} $ \Vert x \Vert=(\displaystyle \sum_{n=1}^{\infty} \Vert x_{n}\Vert^{r})^{\frac{1}{r}}<\infty.$
\end{center} Similarly,\\ $ (\displaystyle \sum_{n=1}^{\infty}\oplus X_{n})_{c_{0}}$
denotes the space of all vector-valued sequences $ x = (x_{n})_{n} $
with $ x_{n}\in X_{n} ~(n\in\mathbb{N}), $ for which $ \displaystyle\lim_{n}\Vert x_{n}\Vert=0, $ endowed with the supreme norm.\\

As an immediate consequence of the Theorems 3.5 and 3.9 in \cite{ccl1}, we can conclude that
the following result.
\begin{corollary}\label{c7}\em
$ \rm{(i)} $ Let $ (X_{n})_{n} $ be a sequence of Banach spaces, $ 1 \leq r < \infty $ and $ 1< p<\infty. $\ Then $ (\sum_{n=1}^{\infty}\oplus X_{n})_{p} \in (RDPP)_{r} $
if and only if $ X_{n}$ has the $ (RDPP)_{r}, $ for each $ n\in \mathbb{N}. $\\
$ \rm{(ii)} $ Let $ (X_{n})_{n} $ be a sequence of Banach spaces.\ Then $ (\sum_{n=1}^{\infty}\oplus X_{n})_{c_{0}} $ has the $ (RDPP)_{1} $ if and only if $ X_{n}$ has the $ (RDPP)_{1}, $ for each $ n\in \mathbb{N}. $\\
$ \rm{(iii)} $ Let $ (X_{n})_{n} $ be a sequence of Banach spaces, $ 1 \leq r < \infty , $ $ 1< p<\infty $ and $ X=(\sum_{n=1}^{\infty}\oplus X_{n})_{p} $ or $ X=(\sum_{n=1}^{\infty}\oplus X_{n})_{c_{0}} . $\ Then $ X$ has the $ (RDP^{\ast}P)_{r} $ if and only if $ X_{n}$ has the $ (RDP^{\ast}P)_{r}, $ for each $ n\in \mathbb{N}. $\
\end{corollary}
The following example shows that there are Banach spaces
$ X $ and $  Y$ such that
$ K_{w^{\ast}}(X^{\ast}, Y ) $ has Pelczy$\acute{n} $ski's property $ (V^{\ast}) $ of order
$ p. $\
\begin{example}\rm\label{e2} Let $ 1< p< \infty. $
Suppose that $ 1 <  r < q < \infty. $\ By Pitt theorem (see \cite{AlbKal}), $L(\ell_{q} ,\ell_{r}) = K(\ell_{q} ,\ell_{r}) .$\ Also, it is known that $ L(\ell_{q} ,\ell_{r})  $ is reflexive (see \cite{k}).\ Therefore,  $K(\ell_{q} ,\ell_{r}) \simeq K_{w^{\ast}}(\ell_{q}^{\ast\ast},\ell_{r})=L_{w^{\ast}}(\ell_{q}^{\ast\ast},\ell_{r})$
has Pelczy$\acute{n} $ski's property $ (V^{\ast}) $ of order
$ p. $\ Hence,  the spaces $ X=\ell_{q}^{\ast} $
 and $ Y=\ell_{r} $ are as desired.
\end{example}
\begin{theorem}\label{t6}
 $ \rm{(i)} $ Suppose that $L_{w^{\ast}}(X^{\ast} ,Y)= K_{w^{\ast}}(X^{\ast} ,Y) .$\ If both $  X$ and $ Y $
have Pelczy$\acute{n} $ski's property $ (V^{\ast}) $ of order
$ p,$ then $ K_{w^{\ast}}(X^{\ast} ,Y) $ has the same property.\\
$ \rm{(ii)} $ Suppose that $L(X ,Y)= K(X ,Y) .$\ If $ X^{\ast} $ and $ Y $ have Pelczy$\acute{n} $ski's property $ (V^{\ast}) $ of order
$ p, $ then $ K(X, Y ) $ has the same property.\
\end{theorem}
\begin{proof}
Since the proofs of (i)  and (ii) are essentially the same, we only present that of (i).\\
$ \rm{(i)} $ Suppose $ X $ and $ Y $ have Pelczy$\acute{n} $ski's property $ (V^{\ast}) $ of order
$ p. $\ Let $ H$ be a
$ p $-$ (V^{\ast}) $ subset of $ K_{w^{\ast}}(X^{\ast} ,Y). $\
For fixed $ x^{\ast}\in X^{\ast} $ the map $ T\mapsto T(x^{\ast}) $ is a bounded
operator from $ K_{w^{\ast}}(X^{\ast} ,Y) $ into $ Y. $\ It is easily verified that continuous linear images of $ p $-$ (V^{\ast}) $ sets are $ p $-$ (V^{\ast}) $ sets.\ Therefore, $ H(x^{\ast}) $
is a $ p $-$ (V^{\ast}) $ subset of $ Y, $ hence relatively
weakly compact.\ For fixed $ y^{\ast}\in Y^{\ast} $ the map $ T\mapsto T^{\ast}(y^{\ast}) $
is a bounded linear operator from
$ K_{w^{\ast}}(X^{\ast} ,Y) $ into $ X. $\ Therefore, $ H^{\ast}(y^{\ast}) $
is a $ p $-$ (V^{\ast}) $ subset of $ X, $ hence relatively weakly compact.\ Hence,  {\rm (\cite[Theorem 4.8]{g17})},
 implies that $ H $ is relatively weakly compact.\
\end{proof}
\begin{corollary}\label{c8}
 $ \rm{(i)} $ Suppose that $L_{w^{\ast}}(X^{\ast} ,Y)= K_{w^{\ast}}(X^{\ast} ,Y) .$\ If both $  X$ and $ Y $
have the $ (RDP^{\ast}P)_{p} ,$
then $ K_{w^{\ast}}(X^{\ast} ,Y) $ has the same property.\\
$ \rm{(ii)} $ Suppose that $L(X ,Y)= K(X ,Y) .$\ If $ X^{\ast} $ and $ Y $ have the $ (RDP^{\ast}P)_{p} ,$  then $ K(X, Y ) $ has the same property.\
\end{corollary}
\begin{remark}\rm\label{r1}\rm
 We know that {\rm (\cite[Theorem 20]{g22})}, shows that $ c_{0}\hookrightarrow  K_{w^{\ast}} (\ell_{2}, \ell_{2}) $
and the identity operator from $ \ell_{2} $ to $ \ell_{2} $ shows that $L_{w^{\ast}} (\ell_{2}, \ell_{2})\neq K_{w^{\ast}} (\ell_{2}, \ell_{2}). $\ In the other word, it is clear that $ (z_{n})=(\sum _{i=1}^{n} e_{i}) $
 is a Dunford-Pettis set which is not relatively weakly compact.\ Therefore,
$ c_{0} $
 does not have the  $ (RDP^{\ast}P) $ and so, does not have the  $ (RDP^{\ast}P) _{p}.$\
Hence,
the space $K_{w^{\ast}} (\ell_{2}, \ell_{2}) $
does
not have the $ (RDP^{\ast}P)_{p} ,$ while $ \ell_{2} $ has this property.\ Hence, the condition $L_{w^{\ast}}(X^{\ast} ,Y)= K_{w^{\ast}}(X^{\ast} ,Y) $ in  Theorem \ref{t6}  and  Corollary \ref{c8}  is necessary.
\end{remark}
\begin{corollary}\label{c9}
$ \rm{(i)} $ Suppose that $ L(X,Y^{\ast}) =K(X,Y^{\ast}).$\ If $ X^{\ast} $ and $ Y^{\ast} $ have  the $ (RDP^{\ast}P)_{p} ,$ then
$ X\widehat{\bigotimes}_{\pi} Y$ does not contain any complemented copy of $ \ell_{1}. $\\
$ \rm{(ii)} $ Suppose that $ Y $ has the Schur property and $ X $ has the $ (RDP^{\ast}P)_{p}. $\
Then $L_{w^{\ast}}(X^{\ast} ,Y)= K_{w^{\ast}}(X^{\ast} ,Y) $ has the $ (RDP^{\ast}P)_{p}. $\\
$ \rm{(iii)} $ Suppose that $ X^{\ast} $ has the Schur property and $ Y $ has the $ (RDP^{\ast}P)_{p}. $\
Then $L(X,Y)= K(X ,Y) $ has the the $ (RDP^{\ast}P)_{p}. $\.\\
$ \rm{(iv)} $ Suppose that $ X $ has the $ (RDP^{\ast}P)_{p}. $\ Then
the space $ \ell_{1}[X] $ of all unconditionally convergent series in $ X $ with norm
$$ \Vert (x_{n}) \Vert=\sup \lbrace \sum \vert x^{\ast}(x_{n})\vert :x^{\ast}\in B_{X^{\ast}}\rbrace,$$
has the same  property.
\end{corollary}
\begin{proof}
$ \rm{(i)} $ By
Corollary \ref{c8}, $ K(X,Y^{\ast}) $ has the $ (RDP^{\ast}P)_{p}. $\  Hence, $L(X,Y^{\ast})=(X\widehat{\bigotimes}_{\pi} Y)^{\ast} $ has the same property.\ Since $ c_{0} $ does not have the  $ (RDP^{\ast}P)_{p},$ $ (X\widehat{\bigotimes}_{\pi} Y)^{\ast} $ does not contain a copy of $ c_{0}. $\ Hence by a result
of Bessaga and Pelczyinski, $ X\widehat{\bigotimes}_{\pi} Y$ does not contain any complemented copy of $ \ell_{1}. $\\
$ \rm{(ii)} $ Let $T\in L_{w^{\ast}}(X^{\ast} ,Y).  $\ Since $ T $ is $ w^{\ast} $-$ w $ continuous, $ T $ is weakly compact.\ Hence $ T $ is compact, since $ Y $ is a Schur space.\ Since $ Y $ has  $ (RDP^{\ast}P)_{p} $ (see Corollary 18 in \cite{g18}),
an application of Corollary \ref{c8} (i) gives that $  K_{w^{\ast}}(X^{\ast} ,Y) $ has the $ (RDP^{\ast}P)_{p}. $\\
$ \rm{(iii)} $ is obvious.\\
$ \rm{(iv)} $ It is known that $ \ell_{1}[X] $ is isometrically isomorphic to $ K(c_{0}, X) $ (see \cite{e2}).\ Since $ X $ has the $ (RDP^{\ast}P)_{p}$ and $ c_{0}^{\ast}=\ell_{1} $ has the Schur property.\ Apply $ \rm{(iii)}. $
\end{proof}

\end{document}